\documentclass{amsart}
        \usepackage{amsmath}
        \usepackage{amsfonts}
        \usepackage{amsthm}
        \usepackage{comment}
        \usepackage{url}
        \usepackage{color}
        \usepackage{enumitem}
        \setlist[enumerate]{font=\normalfont}
        \usepackage{mathtools}
		\usepackage[shortalphabetic]{amsrefs}
        \usepackage{latexsym}
        \usepackage{amssymb}
		\usepackage{graphicx}        
		\usepackage{float}
        \usepackage{youngtab}
        \usepackage[colorinlistoftodos]{todonotes}
		\usepackage[colorlinks=true, allcolors=blue]{hyperref}
		\usepackage{subcaption}

        \usepackage{eucal}
\usepackage{xcolor}
        \theoremstyle{plain}
        \newtheorem*{theorem*}{Theorem}
        \newtheorem{theorem}{Theorem}[section]
        \newtheorem{corollary}[theorem]{Corollary}
        \newtheorem{lemma}[theorem]{Lemma}

        \theoremstyle{definition}
        \newtheorem{definition}[theorem]{Definition}
        \newtheorem{example}[theorem]{Example}
        
        \newtheorem*{example*}{Example}

        \theoremstyle{remark}
        \newtheorem{remark}[theorem]{Remark}
        \newtheorem*{remark*}{Remark}

\newcommand\ZZ{\mathbb{Z}}

\newcommand{\PP}{\mathbb{P}}

\newcommand{\CC}{\mathbb{C}}
\newcommand{\NN}{\mathbb{N}}

\newcommand{\EE}{\mathbb{E}}

\newcommand{\Hom}{\operatorname{Hom}}

\title{On random walks and switched random walks on homogeneous spaces.}
\date{}

\begin{document}
\author{Elvira Moreno}
\address{Elvira Moreno, Computational and Mathematical Sciences,
California Institute of Technology
Department of Computing and Mathematical Sciences
1200 E. California Blvd., MC 305-16
Pasadena, CA 91125-2100} 
\email{emoreno2@caltech.edu}
\author{Mauricio Velasco}
\address{Mauricio Velasco, Departamento de
  Matem\'aticas\\ Universidad de los Andes\\ Carrera 1 No. 18a 10\\ Edificio
  H\\ Primer Piso\\ 111711 Bogot\'a\\ Colombia}
\email{mvelasco@uniandes.edu.co}


\begin{abstract} We prove new mixing rate estimates for the random walks on homogeneous spaces determined by a probability distribution on a finite group $G$. We introduce the switched random walk determined by a finite set of probability distributions on $G$, prove that its long-term behavior is determined by the Fourier joint spectral radius of the distributions and give hermitian sum-of-squares algorithms for the effective estimation of this quantity.
\end{abstract}

\maketitle{}

\section{Introduction}

A person shuffles a deck of $n$ cards. Her shuffling method is specified by a probability distribution $Q$ on the permutation group $S_n$. More concretely, at stage $j=1,\dots, N$ the person takes the deck in some position $v$ in $S_n$ and resuffles it to position $g_jv\in S_n$, where $g_j$ is a random element of $S_n$ selected according to the distribution $Q$, sampled independently of the chosen $g_s$ for $s<j$. The resulting process is called a {\it random walk} in the group $G=S_n$. These processes have been the focus of much work, masterfully explained in the book~\cite{D}. Under common assumptions on the distribution $Q$ such processes approach the uniform distribution on $G$ as $N$ increases (i.e., the deck of cards gets {\it evenly} mixed). A key quantitative question is to determine how quickly this occurs. More precisely, one wishes to bound the total variation distance between the distribution $Q^{N}$ of the process after $N$ stages and the uniform distribution $U$, where the {\it total variation} distance is defined as  
\[\|Q^{N}-U\|_{\rm TV}:=\max_{A\subseteq G}\left|Q^{N}(A)-U(A)\right|\]
More generally, if the group $G$ acts on a set $X$ and $x_0$ is an element of $X$ then the probability distribution $Q$ on $G$ determines a random walk $(h_k)_{k\in \mathbb{N}}$ on $X$ via the formula $h_j:=g_jg_{j-1}\dots g_1\cdot x_0$. 

In the first part of this article (Sections~\ref{Sec: 2} and~\ref{Sec: 3}) we study the behavior of such random walks on sets $X$ where $G$ acts transitively using the modules $\CC X$ over the group ring $\CC G$ which such actions determine (see Section~\ref{Sec: RW_and_modules} for details).  To describe our results precisely we establish some notation. Assume the finite group $G$ has distinct irreducible representations $(V_j,\rho_j)$ for $j=1,\dots, k$, let $\CC X$ be the permutation representation associated to the action of $G$ on $X$, that is, the space with basis given by the symbols $\{e_{x}: x\in X\}$ with the natural action of $G$ (i.e., $e_g\cdot e_x:=e_{g(x)}$) and let $\overline{u}=\frac{1}{|X|}\sum_{x\in X} u_x\in \CC X$ represent the uniform distribution on $X$. Our first result is a bound on the average squared total variation:

\begin{theorem}\label{thm: newUB} Let $Q$ be a probability distribution on $G$ and let $X$ be a $G$-homogeneous space. If $q:=\sum_{g\in G} Q(g)e_g\in \CC G$ then the following inequalities hold:

\[\frac{1}{|X|}\sum_{x\in X}\|q\cdot e_{x}-\overline{u}\|_{\rm TV}^2\leq \frac{1}{4}\sum_{V_j\neq {\rm triv}} m(V_j, \CC X) \|\hat{Q}(\rho_j)\|_{\rm Fb}^2\]

\[\frac{1}{|X|}\sum_{x\in X}\|q\cdot e_{x}-\overline{u}\|_{\rm TV}^2 \geq \frac{1}{4|X|}\sum_{V_j\neq {\rm triv}} m(V_j, \CC X) \|\hat{Q}(\rho_j)\|_{\rm Fb}^2\]
where the matrix $\hat{Q}(\rho_j)$ denotes the value of the Fourier transform of $Q$ in the representation $\rho_j$, $\|A\|_{\rm Fb}^2:={\rm Tr}(AA^*)$ and $m(V_j, \CC X)$ denotes the multiplicity of the irreducible representation $V_j$ in the $\CC G$-module $\CC X$.  If furthermore $\|q\cdot e_{x}\|_2$ is independent of $x\in X$ then for every $x\in X$ we have

\[\|q\cdot e_{x}-\overline{u}\|_{\rm TV}^2\leq \frac{1}{4}\sum_{V_j\neq {\rm triv}} m(V_j, \CC X) \|\hat{Q}(\rho_j)\|_{\rm Fb}^2\]

\[\|q\cdot e_{x}-\overline{u}\|_{\rm TV}^2 \geq \frac{1}{4|X|}\sum_{V_j\neq {\rm triv}} m(V_j, \CC X) \|\hat{Q}(\rho_j)\|_{\rm Fb}^2\]

\end{theorem}

The first part of Theorem~\ref{thm: newUB} implies the existence of {\it deterministic} initial states  
which are especially good and especially bad with respect to mixing, proving that the sum appearing in the Theorems above is a fine estimator of the mixing rate. More precisely, 

\begin{corollary} \label{Cor: good_and_bad} For every integer $N$ there exist initial states $r$ and $s$ in $X_0$ satisfying
\[\|q^N\cdot e_{r}-\overline{u}\|_{\rm TV}^2\leq \frac{1}{4}\sum_{V_j\neq {\rm triv}} m(V_j, \CC X) \|\hat{Q}(\rho_j)^N\|_{\rm Fb}^2\]

\[\|q^N\cdot e_{s}-\overline{u}\|_{\rm TV}^2\geq \frac{1}{4|X|}\sum_{V_j\neq {\rm triv}} m(V_j, \CC X) \|\hat{Q}(\rho_j)^N\|_{\rm Fb}^2\]
\end{corollary}

The second part of Theorem~\ref{thm: newUB} specializes,  when $X=G$, to the celebrated Diaconis-Shahshahani Upper bound Lemma introduced in~\cite{DS1981} but leads to improved estimates of total variation distance whenever $X\neq G$. The reason for this improvement is that the multiplicities appearing in Theorem~\ref{thm: newUB} are $m(V_j,\CC X)$ which are no larger and typically strictly smaller than the multiplicities $\dim(V_j)$ appearing in the Upper Bound Lemma. For instance, this improvement occurs whenever $G$ acts transitively on $X$ and $|X|<|G|$ in the following Corollary

\begin{corollary}\label{Cor: Conjugation_invariant} Suppose $Q$ is a probability distribution on $G$ which is constant on the conjugacy classes of $G$ and let $Q=\sum_{j=1}^k a_j \chi_j$ be its unique representation as a sum of characters. If $q^{(N)}:=\sum_{g\in G} Q^{N}(g)e_g\in \CC G$, then for any $G$-set $X$ and any $x_0\in X$ we have
\[\|q^{(N)}\cdot e_{x_0}-\overline{u}_X\|^2_{\rm TV}\leq \frac{1}{4}\sum_{V_j\neq {\rm triv}}m(V_j,\CC X)\dim(V_j) \left(\frac{a_j|G|}{\dim(V_j)}\right)^{2N}\]
\end{corollary}

In Section~\ref{Sec: RW_tabloids} we apply these  methods to estimate convergence rates for random walks on tabloids, illustrating connections between such estimates and the Kostka numbers of combinatorial representation theory.

The random walks on homogeneous spaces described so far are often easy to simulate on a computer, even in cases when the group $G$ is huge (this occurs for instance whenever the support of the distribution $Q$ is small compared to the size of the group) and therefore give us effective means of approximating the uniform distribution on $G$ by simulating the walk for $N$ stages. Such simulations allow us to understand how {\it typical elements} of the homogeneous space (i.e.,, elements uniformly chosen at random) look like, providing us with a ``statistical" description of a finite group or of a large homogeneous space. Our next result makes this idea precise by giving us a bound on the error resulting from using the random walk at stage $N$ to estimate the true average of a function on $G$. The Theorem provides a key application for the estimates of total variation obtained in Theorem~\ref{thm: newUB}. 
\begin{theorem}\label{Thm: Concentration} Assume $Z_1,\dots, Z_M$ are $M$ independent copies of the random walk on $G$ defined by $Q$ after $N$ stages, $f: G\rightarrow \CC$ is any function with $\max_{g\in G}|f(g)|\leq 1$, and $\epsilon>0$. 
If $\|Q^{N}-U\|_{\rm TV}\leq\epsilon$, then the following inequality holds
\[\PP\left\{\left|\EE_U(f) - \frac{1}{M}\sum_{i=1}^M f(Z_j)\right|\geq \epsilon \right\}\leq 2\exp\left(-\frac{M\left(\epsilon-\|Q^{N}-U\|_{\rm TV}\right)^2}{2}\right)\]
\end{theorem}

A concrete application of Theorem~\ref{Thm: Concentration} for estimating the average features of traveling salesman tours is discussed in Example~\ref{example: TSP_average}.
 
In the second part of this article (Section~\ref{Sec: 4}) we introduce a generalization of the random walk model. The random walk model for card shuffling has a strong assumption, namely that the probability distribution of allowed moves is assumed to be the same at every stage. While this may accurately describe the behavior of a proficient card mixer it may not be adequate for describing many real-life mixing behaviors. A more general approach would be to assume the mixer has a collection of distributions $Q_1,\dots, Q_m$ on $G$ and uses them in some order $w_1w_2\dots $ with $w_i\in [m]$ to shuffle the deck (where the chosen order is perhaps unknown even to the mixer).
We call these more complicated processes {\it switched random walks} by analogy with {\it switched dynamical systems}~\cite{AJ1,AJ2}, which motivated our definition. We ask the following basic questions about the switched random walk defined by a set of distributions $Q_1,\dots, Q_m$:

\begin{enumerate}
\item Does the deck get evenly mixed {\it regardless of the order in which the $Q_i$'s are used}? When the answer is yes, we say that the set $\{Q_1,\dots, Q_m\}$ has the {\it adversarial mixing property}.
\item If $\{Q_1,\dots, Q_m\}$ has the adversarial mixing property then we would like quantitative estimates of how quickly this occurs in the {\it worst case}. By this we mean estimating the maximum total variation distance between the distribution of the process after $N$-stages and the uniform distribution on the permutations of the deck. 
\end{enumerate}

The methods developed for random walks can sometimes be extended to the switched setting. For instance Theorem~\ref{thm: newUB} easily implies the following Corollary, where $X$ is any $G$-set and $x_0\in X$.

\begin{corollary} Suppose $Q_1,\dots Q_m$ are probability distributions on $G$ that are constant on conjugacy classes and let $Q_i=\sum_{j=1}^k a_j^{i} \chi_j$ be their unique representations as sums of characters. For a word $w=w_1w_2\dots w_N$ with $w_i\in [m]$ let $Q^{(w)}$ be the convolution $Q^{(w)}:=Q_{w_1}\ast \dots \ast Q_{w_N}$. If $q^{w}:=\sum_{g\in G}Q^{(w)}(g)e_g$, then the following inequality holds 
\[\|q^{w}\cdot e_{x_0}-\overline{u}_X\|^2_{\rm TV}\leq \frac{1}{4}\sum_{V_j\neq {\rm triv}}m(V_j,\CC X)\dim(V_j)\left(\frac{|G|}{\dim(V_j)}\right)^{2N}\left(\prod_{i=1}^N a_j^{i}\right)^2\]
\end{corollary}

The assumption that the $Q_i$ are constant on conjugacy classes makes the dynamics much simpler because, via Fourier transform, reduces it to the problem of understanding the behavior of products of {\it commuting} matrices. To study the general non-commutative case we introduce the {\it Fourier joint spectral radius} of a set of distributions $Q_1,\dots, Q_m$ on $G$ relative to a $G$-set $X$, defined as
\[\overline{\omega}_X\left(Q_1,\dots,Q_m\right):=\max_{\rho_j\in \CC X, \rho_j\neq {\rm triv}} \left(jsr\left(\widehat{Q_1}(\rho_j),\dots ,\widehat{Q_m}(\rho_j)\right)\right)\]
where the maximum is taken over the irreducible representations $\rho_j$ of $G$ appearing with non-zero mutiplicity in $\CC X $ and the symbol ${\rm jsr}$ denotes the joint spectral radius of a set of matrices (see Section~\ref{Sec:Fourier-jsrs} for preliminaries on joint spectral radii).
Our next result proves that the Fourier spectral radius captures the asymptotic worst case behavior of the total variation distance to the uniform distribution, 

\begin{theorem}\label{Thm: Fourier_jsrs} For a word $w=w_1w_2\dots w_N$ with $w_i\in [m]$ let $Q^{(w)}$ be the convolution $Q^{(w)}:=Q_{w_1}\ast \dots \ast Q_{w_N}$. If $q^{w}:=\sum_{g\in G}Q^{(w)}(g)e_g\in \CC G$, then the following equality holds for every $G$-set $X$,
\[\lim_{N\rightarrow \infty} \left(\max_{x_0, w:|w|=N}\|q^{(w)}\cdot e_{x_0} -\overline{u}\|_{\rm TV}\right)^{\frac{1}{N}}=\overline{\omega}_X\left(Q_1,\dots, Q_m\right)\]
where the maximum is taken over all words $w$ of length $N$ and initial states $x_0\in X$. Furthermore, determining whether a set of distributions $Q_1,\dots, Q_m$ has the adversarial mixing property on $X$ is equivalent to determining whether $\overline{\omega}_X\left(Q_1,\dots, Q_m\right)<1$.
\end{theorem}

The computation of the jsr of a set of matrices is a rather difficult task and we expect this difficulty to also extend to Fourier jsrs. For instance it is known that it is {\it undecidable} whether the jsr of a pair of matrices is at most one~\cite{BT} and it is not known whether checking if the strict inequality holds is decidable. It is therefore a question of much interest to device methods for estimating joint spectral radii (even with the knowledge that they are bound to fail in some cases). Recent work by Ahmadi, de Klerk and Hall~\cite{PN} gives a hierarchy of polynomial norms that can be used to produce a sequence of converging upper bounds to the jsr of a set of matrices. In the final section (Section~\ref{Sec: Estimation_Fourier_jsrs}) of this article we extend their results to norms on complex vector spaces that are expressible as hermitian sums-of-squares, allowing us to estimate Fourier spectral radii via Hermitian semidefinite programming.

{\bf Acknowldegments.} We wish to thank Michael Hoegele, Mauricio Junca and Pablo Parrillo for useful conversations during the completion of this work. Mauricio Velasco was partially supported by proyecto INV-2018-50-1392 from Facultad de Ciencias, Universidad de los Andes and by Colciencias grant EXT-2018-58-1548.

\section{Preliminaries}
\label{Sec: 2}
\subsection{Representation theory of finite groups} Throughout the article $G$ will denote a finite group. By a {\it representation} of $G$ we mean a pair $(V,\rho)$ where $V$ is a finite dimensional vector space over the complex numbers and  $\rho: G\rightarrow GL(V)$ is a group homomorphism. A {\it morphism} between the representations $(V_1,\rho_1)$ and $(V_2,\rho_2)$ is a linear map $\psi: V_1\rightarrow V_2$ with the property that $\psi\circ\rho_1(g) = \rho_2(g)\circ \psi$ for every $g\in G$. A subspace $W\subseteq V$ is an {\it invariant subspace} if $\rho(g)(W)\subseteq W$ for all $g\in G$. A representation $(V,\rho)$ is {\it irreducible} if its only invariant subspaces are $0$ and $W$. An {\it invariant inner product} on a representation $V$ is a hermitian inner product satisfying $\langle \rho(g)u,\rho(g)(v)\rangle = \langle u,v\rangle$ for all $g\in G$ and $u,v\in V$. Throughout the article we will use the following fundamental facts about complex representations of finite groups (see~\cite[Chapter 1]{FH} for proofs):

\begin{enumerate}
\item Every finite group $G$ has finitely many pairwise non-isomorphic irreducible representations, which we will denote with $V_1,\dots, V_k$.
\item The irreducible representations are the builing blocks of all representations in the sense that for any representation $(\Lambda,\rho)$ there is an isomorphism of representations
\[ \Lambda\cong \bigoplus_{j=1}^k V_j^{m_j}\]
where the integers $m_j$, called multiplicities are uniquely specified. We write $m(V_j,V):=m_j$.

\item Every irreducible representation has an invariant inner product, unique up to multiplication by a positive real number and we fix a basis $B_j$ for each $V_j$, orthonormal with respect to this product. In this basis the matrices $[\rho(g)]_{B_j}$ are unitary. If $\langle, \rangle$ is an invariant inner product on a representation $\Lambda$ then there is an orthonormal basis for $\Lambda$, compatible with the isomorphism in $(2)$ with respect to which the maps of the representation are simultaneously block-diagonal of the form
\[ [\rho_\Lambda(g)]_B = \bigoplus_{j=1}^k \left([\rho_{V_j}(g)]_{B_j} \otimes I_{m_j}\right)\] 
where $I_{m_j}$ is an $m_j\times m_j$ identity matrix.
\item The {\it character} of a representation $V$ is a function $\chi_V: G\rightarrow \CC$ given by $\chi_V(g)={\rm Tr}(\rho(g))$. Characters are constant functions in the conjugacy classes of $G$ and the characters of the irreducible representations $V_j$ form an orthonormal basis for such functions (under the inner product $\langle f,h\rangle := \sum_{g\in G} f(g)\overline{h(g)}/|G|$).

\end{enumerate}

\subsection{The group ring and the Fourier transform}

The {\it group algebra} of $G$, denoted by $\CC G$ is the complex vector space with basis given by the symbols $\{e_g: g\in G\}$ endowed with the multiplication $e_g\cdot e_h:= e_{g\cdot h}$ where the dot in the right hand side is the product in $G$. This is an associative and generally non-commutative algebra of dimension $|G|$. 
If $(\Lambda,\rho)$ is a representation of $G$ then 
there is a linear map $\phi: \CC G\rightarrow \Hom(\Lambda,\Lambda)$ which sends $\sum a_g e_g$ to the map sending $w\in \Lambda$ to $\sum a_g\rho(g)(w)$. This map transforms the product in $\CC G$ into the composition of linear maps and makes $\Lambda$ into a $\CC G$-module. It is easy to see that there is a correspondence between $\CC G$-modules and representations of $G$. In particular the group algebra is itself a representation of $G$ via left-multiplication defining $\rho_{\CC G}(g)(e_h)=e_g\cdot e_h$. Three very useful facts about this representation are:
\begin{enumerate}
\item The representation $\CC G$ is isomorphic to the representation $\CC[G]$ defined as the collection of complex-valued functions $f:G\rightarrow \mathbb{C}$ endowed with the contragradient action $\rho^*(g)f(x)=f(g^{-1}x)$. We will use this isomorphism throughout. It is given explicitly by mapping a function $Q:G\rightarrow \CC$ to the element $q:=\sum_{g\in G} Q(g)e_g$. 

\item If $q_1$ and $q_2$ are the elements of $\CC G$ corresponding to functions $Q_1$ and $Q_2$ then their product $q_1q_2\in \CC G$ corresponds to the {\it convolution} $Q_1\ast Q_2$ of $Q_1$ and $Q_2$ defined as
\[Q_1\ast Q_2(h) = \sum_{g\in G} Q_1(hg^{-1})Q_2(g).\] 
\item There is an isomorphism of representations $\CC G\cong \bigoplus_{j=1}^k V_j^{\rm dim(V_j)}$ and in particular $m(V_j,\CC G)=\dim(V_j)$.

\item There is an isomorphism of algebras $\phi: \CC G\rightarrow \bigoplus_{j=1}^k \Hom(V_j,V_j)$ called the {\it Fourier transform} which is the map sending a function $f$ to the map $\bigoplus_{j=1}^k \hat{f}(\rho_j)$ where 
\[\hat{f}(\rho_j):=\sum_{g\in G} f(g)[\rho_j(g)]_{B_j}.\]
see~\cite[Exercise 3.32]{FH} for basic properties of the Fourier transform.
\end{enumerate}

\section{Random walks on homogeneous spaces and modules over group rings}
\label{Sec: 3}

A {\it homogeneous space} for $G$ is a set $X$ endowed with a transitive action of $G$. In this section we study random walks induced on $X$ by random walks on $G$. More precisely, any probability distribution $Q$ on $G$ and initial state $x_0\in X$ defines a stochastic process $(h_k)_{k\geq 1}$ on $X$, as follows:

\begin{enumerate}
\item Let $g_1,g_2,\dots$ be a sequence of independent and identically distributed random elements of $G$ having distribution $Q$. 
\item Define the random variable $h_j:=g_j \dots g_2g_1(x_0)$.
\end{enumerate}

There are two natural questions about the process $h_N$:

\begin{enumerate}
\item What is the distribution of $h_N$?
\item How does the distribution of $h_N$ vary as $N$ grows? Since the action of $G$ in $X$ is transitive, it should be fairly common (for instance when $Q$ assigns sufficiently large probability to all elements of $G$) that the process mixes $X$ evenly. More quantitatively we ask: What is the total variation distance between the distribution of $h_N$ and the uniform distribution on $X$?
\end{enumerate}

We will address the questions above using the module $\CC X$ over the group ring $\CC G$ defined by an action, borrowing the maxim of modern commutative algebra of placing a greater emphasis on modules. The material is organized as follows: Section~\ref{Sec: RW_and_modules} introduces the basic theory, Section~\ref{Sec: Mixing_bounds} contains the resulting convergence bounds and clarifies their relationship with previous work, Section~\ref{Sec: RW_tabloids}  discusses random walks on tabloids, illustrating how tools from combinatorial representation theory can be used for estimating mixing rates. Finally Section~\ref{Section: Concentration} discusses the application of mixing rates for obtaining "statistical" descriptions of homogeneous spaces and a detailed analysis of the space of traveling salesman tours.

For use throughout the Section, recall that the {\it total variation} distance between two probability distributions $P,Q$ on $X$ is given by
\[\|P-Q\|_{\rm TV} := \max_{A\subseteq X} |P(A)-Q(A)|=\frac{1}{2}\sum_{x\in X}|P(x)-Q(x)|.\].

\subsection{Random walks and modules over the group ring}\label{Sec: RW_and_modules}

Let $\CC X$ be a vector space with basis given by the set of symbols $S:=\{e_x: x\in X\}$. The action of $G$ on $X$ makes $\CC X$ into a $\CC G$-module via the map $\phi: \CC G\rightarrow Hom(\CC X, \CC X)$ given by $\phi(e_g)(e_x) := e_{g(x)}$. We will denote this action by $e_g\cdot e_x$. The following proposition shows that the module structure can be used to compute the distributions of our random walks.

\begin{lemma} If $Q$ is a probability distribution on $G$ and $T$ is a probability distribution on $X$ then the equality
\[\sum_{x\in x} \PP\{W=x\}e_x = q\cdot t\]
holds, where $q:=\sum_{g\in G} Q(g)e_g$, $t:=\sum_{x\in X} T(x)e_x$ and $W:=g(z)$ is the random variable obtained by choosing $g\in G$ and $z\in X$ independently with distributions $Q$ and $T$ respectively. In particular the distribution of $h_N$ is given by $q^{N}\cdot e_{x_0}\in \CC X$. 
\end{lemma}
\begin{proof} The independence between $g$ and $z$ implies the following equality for any $\alpha\in X$ \[\PP\{W=\alpha\} = \sum_{x\in X} \sum_{g\in G: g(x)=\alpha} Q(g)T(x).\]
It follows that
\[\sum_{\alpha\in X} \PP\{W=\alpha\} e_{\alpha} =\sum_{\alpha\in X} \sum_{x\in X} \sum_{g\in G: g(x)=\alpha} Q(g)H(x) e_{\alpha} = \sum_{x\in X} \sum_{g\in G} Q(g)H(x)e_{g(x)} = q\cdot t\]
The last claim follows from the associativity relation $(q_1q_2)\cdot h= q_1\cdot\left(q_2\cdot h\right)$ which holds for all $q_1,q_2\in \CC G$ and $h\in \CC X$ because $\CC X$ is a $\CC G$-module.
\end{proof}

The previous Lemma is useful because it allows us to use the representation theory of $\CC X$ to compute the probability distribution of $h_N$. Henceforth, we endow $\CC X$ with the hermitian inner product which satisfies $\langle e_x,e_y\rangle =\delta_{xy}$. This product is invariant because $G$ acts on $X$ by permutations.

\begin{lemma}\label{lem: orto} Let $Q$ be any complex-valued function on $G$ and $q:=\sum_{g\in G}Q(g)e_g$. There exists an $|X|\times |X|$ unitary matrix $W$ such that 
\[W [\phi(q)]_S W^* = \bigoplus_{j=1}^k \hat{Q}(\rho_{j})\otimes I_{m(V_j,\CC X)}\] 
where $\widehat{Q}$ denotes the Fourier transform of $Q$.
\end{lemma}
\begin{proof} Since the inner product we defined in $\CC X$ is $G$-invariant we can use it to construct an orthonormal basis $B$ for $\CC X$ compatible with the decomposition of $\CC X = \bigoplus V_i^{m(V_i,\CC X)}$ as a representation of $G$. Letting $W$ be the change of basis matrix from the basis $S=\{e_x:x\in X\}$ to the basis $B$ we see that $W$ is unitary and that for every $g\in G$ the equality
\[W[\rho_{\CC X}(g)]_S W^* = \bigoplus_{j=1}^k [\rho_j(g)]_{B_j}\otimes I_{m(V_i,\CC X)}\] 
holds.
Since $\phi$ is linear we conclude that
\[W\phi(q)W^*=\sum_{g\in G} Q(g)\bigoplus_{j=1}^k [\rho_j(g)]_{B_j}\otimes I_{m(V_i,\CC X)}=\]
\[=\bigoplus_{j=1}^k \left(\sum_{g\in G} Q(g)[\rho_j(g)]_{B_j}\right)\otimes I_{m(V_i,\CC X)}\]
which agrees with the claimed formula by definition of Fourier transform.
\end{proof}

\begin{remark}\label{Rmk: Frobenius_Reciprocity}
In order to use Lemma~\ref{lem: orto} it is extremely useful to be able to decompose $\CC X$ into irreducibles. This process is  often simplified by the following two facts:

\begin{enumerate}
\item If $x_0\in X$ is any point and $H\subseteq G$ is the stabilizer of $x_0$ then the permutation module $\CC X$ is isomorphic to the {\it induced representation} of $G$ defined by the trivial representation of $H$~\cite[Example 3.13]{FH}

\item In particular, the multiplicities with which the irreducible representations $V_j$ appear in $\CC X$ are determined by the Frobenius reciprocity theorem~\cite[Corollary 3.20]{FH}, which implies that
\[m(V_j,\CC X) = \frac{1}{|H|}\sum_{h\in H} \chi_{V_j}(h)\]
\end{enumerate}
\end{remark}

\subsection{Mixing rates}\label{Sec: Mixing_bounds}

For any $G$-homogeneous space $X$ we let $\overline{u}\in \CC X$ be the element corresponding to the uniform distribution, that is, $\overline{u}:=\frac{1}{|X|}\sum_{x\in X} e_x$. Note that $\overline{u}_G\cdot e_{x_i}=\overline{u}$ for every $x_i\in X$. We endow $\CC X$ with the total variation norm $\|h\|_{\rm TV}:=\frac{1}{2}\sum_{x\in X} |h(x)|$ and endow $\Hom(\CC X, \CC X)$ with the Frobenius norm $\|A\|_{\rm Fb}:={\rm Tr}(AA^*)$. We are now ready to prove Theorem~\ref{thm: newUB}, our main tool for establishing convergence estimates on homogeneous spaces. The key point of the proof is that while the total variation norm is not unitarily invariant, it can be bounded on average by a unitarily invariant norm allowing us to choose the convenient orthonormal bases for our operators coming  from Lemma~\ref{lem: orto}.

\begin{proof}[Proof of Theorem~\ref{thm: newUB}] 
The equality
\[\sum_{x\in X} \|q\cdot e_x-\overline{u}\|_2^2=\|\phi(q)-\phi(\overline{u}_G)\|_{\rm Fb}^2\]
holds since both sides equal the sum of the squares of the entries of the matrix $\phi(q)-\phi(\overline{e_G})$.
Since the Frobenius norm is unitarily invariant we can compute the right hand side of this equality in any orthonormal basis. In particular, using the basis from Lemma~\ref{lem: orto} we conclude that \[\|\phi(q)-\phi(\overline{u}_G)\|_{\rm Fb}^2=\sum_{j=1}^k m(V_j,\CC X) \|\widehat{Q-U}(\rho_j)\|_{\rm Fb}^2\]
Where $U$ is the uniform disribution on $G$. For any probability distribution $Q$ on $G$ we know that $\hat{Q}({\rm trivial})=1$ and for the uniform distribution $U$ we know that
$\hat{U}(\rho_j)=0$ for all $\rho_j\neq {\rm triv}$. We thus conclude that the following equality holds

\begin{equation}
\label{eq: l2}
\sum_{x\in X} \|q\cdot e_x-\overline{u}\|_2^2=\sum_{V_j\neq {\rm triv}} m(V_j,\CC X) \|\widehat{Q}(\rho_j)\|_{\rm Fb}^2.
\end{equation}

The Cauchy-Schwarz inequality and the fact that $\|\bullet\|_2\leq \|\bullet\|_1$ imply that for every $x\in X$ we have
\begin{equation}
\label{Ineq: basic}
 \frac{1}{4}\|q\cdot e_x-\overline{u}\|_2^2 \leq \|q\cdot e_x-\overline{u}\|_{\rm TV}^2\leq \frac{1}{4}|X|\|q\cdot e_x-\overline{u}\|_2^2.
\end{equation}
Averaging the inequalities in~(\ref{Ineq: basic}) over $X$ we obtain
\[\frac{1}{4|X|}\sum_{x\in X}\|q\cdot e_x-\overline{u}\|_2^2\leq  \frac{1}{|X|}\sum_{x\in X} \|q\cdot e_x-\overline{u}\|_{\rm TV}^2\leq \frac{1}{4}\sum_{x\in X} \|q\cdot e_x-\overline{u}\|_2^2.\]
and combining these inequalities with identity~(\ref{eq: l2}) we complete the proof of the two inequalities in our main claim.

If furthermore $\|q\cdot e_x\|_2^2$ is independent of $x$ then the same is true of $\|q\cdot e_x-\overline{u}\|_2^2$ since this quantity equals $\|q\cdot e_x\|^2-2\langle q\cdot e_x,\overline{u}\rangle + \|\overline{u}\|_2^2$ which is independent of $x$ because $\langle q\cdot e_x,\overline{u}\rangle = 1/|X|$ for all $x$. As a result, for every $x\in X$ we have
\[\|q\cdot e_x-\overline{u}\|_2^2=\frac{1}{|X|}\sum_{y\in X}\|q\cdot e_y-\overline{u}\|_2^2\]
and we can replace the leftmost and rightmost terms in~(\ref{Ineq: basic}) for averages, yielding
\[
\frac{1}{4|X|}\sum_{y\in X}\|q\cdot e_y-\overline{u}\|_2^2 \leq \|q\cdot e_x-\overline{u}\|_{\rm TV}^2\leq \frac{1}{4}\sum_{y\in X}\|q\cdot e_y-\overline{u}\|_2^2
\]
which, combined with~(\ref{eq: l2}) completes the proof.
\end{proof}

\begin{remark} There are two cases where the condition that $\|q\cdot e_x\|_2$ be independent of $X$ occurs automatically because $q\cdot e_{x_2}$ is  obtained from $q\cdot e_{x_1}$  by rearranging the order of the coefficients for every $x_1,x_2\in X$. This happens
\begin{enumerate}
\item If $X=G$ because the set of coefficients of $q\cdot e_{x}$ for any $x$ is exactly the set of values of $Q$.
\item If $Q$ is constant on conjugacy classes of $G$. This holds because the equality $q\cdot e_{x_2}=\sum_{g\in G} Q(g)e_{g(x_2)}=\sum_{y\in X} c_ye_y$ implies that $c_y$ equals the sum of the $Q$-probabilities of the set $A$ of $g\in G$ with $g(x_2)=y$. It follows that if $rx_1=x_2$ then the conjugates $r^{-1}gr$ for $g$ in $A$ are the group elements which map $x_1$ to $r^{-1}y$. Since $Q$ is conjugation invariant, we have that $c_y$ is the coefficient of $r^{-1}y$ in $q\cdot e_{x_1}$. We conclude that $q\cdot e_{x_1}$ is a permutation of $q\cdot e_{x_2}$.
\end{enumerate}
\end{remark}

\begin{remark} More generally, it can be shown that if $T$ is a subgroup of $G$ which acts transitively on $X$ and $Q$ is a probability distribution which is constant on the conjugacy classes of $T$ then:
\begin{enumerate}
\item The convolution $Q^{(N)}$ is also constant on the conjugacy classes of $T$ and
\item  $\|q\cdot e_{x}\|_2$ is independent of $x\in X$.
\end{enumerate}
So in this case the bound from Theorem~\ref{thm: newUB} holds for every initial state $x_0\in X$.
\end{remark}

\begin{remark} The proof of Theorem~\ref{thm: newUB} uses the same approach as that of the celebrated Upper Bound Lemma~\cite[pag. 24]{D} for random walks on a group $G$, the Lemma was introduced in the early 80's and is still a key tool in the state-of-the-art analysis of Markov chains (see for instance~\cite{M1,M2}). 

The use of averaging allows us to extend the result to arbitrary homogeneous spaces, increasing the range of applications, and to improve the coefficients in the inequality by replacing $\dim(V_j)$ with the typically smaller $m(V_j, \CC X)$.
\end{remark}

\begin{remark} Theorem~\ref{thm: newUB} should be compared with the Upper Bound Lemma for homogeneous spaces (UBLH) from~\cite[pag. 53]{D}. 
The UBLH applies to random walks defined by distributions $Q$ on $G$ which are right invariant under the subgroup $H\leq G$ which stabilizes $x_0\in X$ (i.e.,, the distributions are forced to satisfy $Q(gh)=Q(g)$ for all $h\in H$) and the bound depends on the Fourier transforms of the induced distributions on $X$ and not on the Fourier transforms of the original distributions. Due to the restriction to right-$H$-invariant distributions and the presence of Fourier transforms of induced distributions, the UBLH has a smaller range of applicability than Theorem~\ref{thm: newUB}. On the other hand the UBLH gives inequalities valid for every initial state $x_0$, albeit at the cost of using larger constants than the $m(V_j, \CC X)$ of Theorem~\ref{thm: newUB}. It follows that the Theorems are strictly incomparable and that it may be more convenient to use one or the other, depending on the intended application.
\end{remark}

\begin{proof}[Proof of Corollary~\ref{Cor: good_and_bad}] Theorem~\ref{thm: newUB} provides us with lower and upper bounds for the average of the squared total variation over the starting points of $X$. The average of a set of real numbers is at least the smallest one and at most the largest one proving the existence of the good and bad initial states $r$ and $s$ in $X$.\end{proof}

\begin{remark} Theorem~\ref{thm: newUB} and Markov's inequality imply that for every $\alpha\geq 0$ the inequality
\[\left|\left\{x\in X: \|q^Ne_{x}-\overline{u}\|_{\rm TV} \geq \alpha\right\}\right| \leq \frac{|X|}{4\alpha^2}\sum_{V_j\neq {\rm triv}} m(V_j, \CC X) \|\hat{Q}(\rho_j)^N\|_{\rm Fb}^2\]
holds, allowing us to prove that most (and even all, when the right-hand side is $<1$) initial states mix well. In the special case when $Q$ is right-invariant under the stabilizer of a point $x_0\in X$ this inequality is weaker than the UBLH, which provides a bound for every initial state. However, the inequality above applies to arbitrary (not necessarily right-invariant) distributions $Q$.
\end{remark}

\subsection{Example: Random walks on tabloids}\label{Sec: RW_tabloids}

Fix positive integers $n,a,b$ with $a\geq b$ and $a+b=n$. Suppose we have a set of $n$ of cards and that these are placed face up forming a single row. We permute the cards by independently sampling permutations according to a fixed distribution $Q$ on $S_n$, and acting with these permutations on the row of cards. After $N$ stages we split the row of cards into two disjoint sets $A$ and $B$ of sizes $a$ and $b$ consisting of the first $a$ cards and the remaining $b$ cards respectively (reading the row of cards from left to right). We ask: How near-uniform is the set $A$ or equivalently: how random is the set partition $(A,B)$? In this Section we define random walks on tabloids, which generalize this problem, and discuss tools suitable for their analysis. 

\subsubsection{Preliminaries: partitions, tableaus and tabloids}
\label{sec: tabloids}
A partition of a positive integer $n$ is a nonincreasing sequence $\lambda_1\geq \lambda_2\geq \dots\geq \lambda_k$ with $n=\sum \lambda_i$. Partitions are partially ordered by the {\it dominance ordering} defined as $\lambda\leq \mu$ if $\sum_{i\leq j} \lambda_i \leq  \sum_{i\leq j}\mu_i$ for all $j=1,\dots, n$. A {\it tableaux} with shape $\lambda$ is a finite collection of boxes arranged in left-justified rows of sizes $\lambda_1,\dots, \lambda_k$ containing the integers $1,\dots, n$ without repetitions. Two tableaus of the same shape $\lambda$ are {\it row-equivalent} if the sets of elements in each of their corresponding rows coincide. A {\it tabloid} is a row-equivalence class of tableaus.

\begin{example} The sequence $\lambda:=(3,3,2,1)$ is a partition of $9$. The following two tableaus are row equivalent and therefore members of the same tabloid.
\begin{center}
\young(123,456,78,9)\text{ , } \young(213,564,87,9)
\end{center}
This tabloid is encoded by the ordered set partition $\left(\{1,2,3\}, \{4,5,6\}, \{7\}, \{8\}\right)$ which keeps track of the set of elements in each row.
\end{example}

A {\it generalized tableaux } of shape $\lambda$ is a tableaux $T$ of shape $\lambda$ filled with elements of $\{1,\dots, n\}$ where repetitions are allowed. The content of such a tableaux is a vector $(\mu_1,\dots,\mu_n)$ where $\mu_i$ is the number of copies of the integer $i$ appearing in $T$ for $i=1,\dots, n$. A {\it semi-standard tableaux} is a generalized tableaux where the labels are weakly increasing along rows and strictly increasing along columns.

\begin{example} If we fix the content $\mu=(2,2,1)$ then the set of all semi-standard tableaus of shapes $(3,2)$ and $(4,1)$ with content $\mu$ are given by
\begin{center}
\young(112,23)\text{ , } \young(113,22)
\text{ , }\young(1122,3)\text{ , } \young(1123,2)
\end{center}
\end{example}

\subsubsection{Conjugation-invariant distributions and random walks on tabloids}

Continuing with the example introduced at the beginning of Section~\ref{Sec: RW_tabloids}, any partition $\lambda$ of $n$ having $k$ parts allows us to partition our row of $n$ cards into $k$ sets $(A_1,\dots, A_k)$ where $A_1$ consists of the first $\lambda_1$ cards, the set $A_2$ consists of the next $\lambda_2$ cards (those in positions $\lambda_1+1,\dots,\lambda_1+\lambda_2$ along the row), $A_3$ consists of the next $\lambda_3$ cards, etc. and we can ask: how near-uniform is the resulting set partition $(A_1,\dots, A_k)$ after $N$ stages of our random walk? In this Section we address this problem when the distribution $Q$ is constant on conjugacy classes by applying the tools introduced in the article. Our main result is Corollary~\ref{Example: main}, which gives bounds on the mixing rate of the process described in the first paragraph, that is, when $\lambda$ is any partition with at most two parts. 

To begin the analysis, note that if $X$ denotes the set of tabloids of shape $\lambda$ with the natural action of $S_n$ by permutations (see Section~\ref{sec: tabloids} for basic definitions) then the process above coincides with the random walk on the homogeneous space $X$ defined by the probability distribution $Q$ on $S_n$. The corresponding permutation module $\CC X$ is well-known and plays a distinguished role in Young's construction of the irreducible representations of the group $S_n$ (see \cite[Chapter 2]{Sagan}). It is common in the literature to refer to these modules as {\it the permutation modules $M^{\lambda}$} and we will do so throughout this Section. The following Theorem~\cite[Theorem 2.11.2]{Sagan} describes their structure, where $S^{\lambda}$ denotes the irreducible representation of $S_n$ corresponding, via Young's construction, to the partition $\lambda$ (see ~\cite[Definition 2.3.4]{Sagan} for an explicit description of $S^{\lambda}$).

\begin{lemma}[Young's rule] \label{Lem: YoungRules} Let $\mu$ be a partition of $n$. The following isomorphism of representations holds
\[M^{\mu} \cong \bigoplus_{\lambda: \lambda\geq \mu}  \left(S^{\lambda}\right)^{\bigoplus K_{\lambda\mu}}\]
where the sum runs over the partitions $\lambda$ of $n$ with $\lambda\geq \mu$ in the dominance ordering of partitions and $K_{\lambda\mu}$ is the Kostka number of $(\lambda,\mu)$, that is, the number of semi-standard tableaus of shape $\lambda$ and content $\mu$.
\end{lemma}

In order to understand the behavior of Markov chains defined by conjugation-invariant probability distributions, we need to express such distributions as sums of characters. To this end, we will use orthogonality of characters together with Frobenius remarkable character formula~\cite[Theorem 4.10]{FH}, which gives a combinatorial description of the characters of $S^{\lambda}$. Recall that the conjugacy class of a permutation $\sigma\in S_n$ is specified by its cycle type, namely the sequence $(n_1,\dots, n_n)$ where $n_j$ equals the number of $j$-cycles appearing in the unique decomposition of $\sigma$ as a product of disjoint cycles. 

\begin{lemma}[Frobenius character formula] If $\lambda=\lambda_1\geq \lambda_2\geq \dots \lambda_n\geq 0$ is a partition of $n$ then the value of the character of $S^{\lambda}$ in the conjugacy class $(n_1,\dots, n_n)$ is given by the coefficient of the monomial $x^{\ell(\lambda)}$ in the polynomial $\Delta\cdot P_n\in \CC[x_1,\dots, x_n]$ where
\[P_n(x):=\prod_{j=1}^n (x_1^j+\dots+x_n^j)^{n_j}\text{ , 
}\Delta:=\prod_{1\leq i<j\leq n} \left(x_i-x_j\right).\]
and $\ell(\lambda)=\left(\lambda_1+n-1,\lambda_2+ n-2,\dots,\lambda_{k}+n-k,\dots,\lambda_{n-1}
+1, \lambda_n\right)$.
\end{lemma}

\begin{lemma} \label{lem: 2rows} Let $\lambda=a\geq b$ be a partition of $n$. If $C_k$ denotes the conjugacy class of a $k$-cycle in $S_n$ then
\[\chi_{S^{\lambda}}(C_k) = \left[\binom{n-k}{a}-\binom{n-k}{a+1}\right] + \left[\binom{n-k}{b}-\binom{n-k}{b-1}\right]\]
and furthermore 
\[\dim(S^{\lambda}) = \binom{n}{a}-\binom{n}{a+1}=\binom{n}{b}-\binom{n}{b-1}\]
\end{lemma}
\begin{proof} Since $\Delta$ is the determinant of the Vandermonde matrix with $i,j$ entry given by $x^j_{n-i+1}$ we know that 
\[\Delta = \sum_{\sigma\in S_n} {\rm sgn}(\sigma) x_n^{\sigma(1)-1}\dots x_1^{\sigma(n)-1}\]
and in particular, only two terms have exponents which are componentwise smaller than $\ell(\lambda) = (a+n-1,b+n-2,n-3,n-4,\dots, 1,0)$, namely 
\[
 x_1^{n-1}x_2^{n-2}x_{3}^{n-3}\dots x_{n-2}^2x_{n-1}^1x_n^0\text{ and }- x_1^{n-2}x_2^{n-1}x_{3}^{n-3}\dots x_{n-2}^2x_{n-1}^1x_n^0.\]
Since $P_{C_k}(x)=A(x)(x_1^k+\dots+x_n^k)$, where $A(x)=(x_1+\dots+x_n)^{n-k}$, we conclude from the Frobenius' character formula and the observation from the previous paragraph that $\chi_{S^{\lambda}}(C_k)$ is given by
\[ [P(x)]_{(a,b)} - [P(x)]_{(a+1,b-1)}\]
where we have removed the $n-2$ trailing zeroes in our notation for exponents. Since $P(x)$ factors, this quantity equals
\[ [A(x)]_{(a-k,b)}+[A(x)]_{(a,b-k)} - \left([A(x)]_{(a-k+1,b-1)} + [A(x)]_{(a+1,b-k-1)}\right)\]
each of these coefficients can be computed by the multinomial Theorem, yielding
$\binom{n-k}{b} + \binom{n-k}{a}-\left(\binom{n-k}{b-1}+\binom{n-k}{a+1}\right)$. Similarly, the dimension of $S^{\lambda}$ is given by the value of its character at the identity element, given by
\[[(x_1+\dots+x_n)^n]_{(a,b)}-[(x_1+\dots+x_n)^n]_{(a+1,b-1)}\]
proving the claim.
\end{proof}

The following Corollary estimates the rate of convergence to the uniform distribution of the pure-cycle random walk on the set of disjoint pairs $(A,B)$ of sizes $a$ and $b$ respectively, with $A\cup B=[n]$.

\begin{corollary}\label{Example: main} Let $Q$ be the probability distribution which samples $k$-cycles uniformly. If $\CC X$ is the permutation module corresponding to a partition $\lambda = a\geq b$ with $a+b=n$ and $q^{(N)}:=\sum_{g\in G} Q^{N}(g)e_g\in \CC G$ then for any $x_0\in X$ the quantity  $\|q^{(N)}\cdot u_{x_0}-\overline{u}_X\|^2_{\rm TV}$ is bounded above by
\[\frac{1}{4}\sum_{t=1}^b \frac{\left(\left[\binom{n-k}{n-t}-\binom{n-k}{n+1-t}\right] + \left[\binom{n-k}{t}-\binom{n-k}{t-1}\right]\right)^{2N}}{\left(\binom{n}{t}-\binom{n}{t-1}\right)^{2N-1}}\]
\end{corollary}
\begin{proof} By Lemma~\ref{Lem: YoungRules} the decomposition of $\CC X$ into $S_n$-irreducibles is given by $\CC X = \bigoplus_{t=0}^b S^{(a+b-t,t)}$ and in particular, no representation appears with multiplicity greater than one. For notational convenience we write $V_t:=S^{(a+b-t,t)}$ throughout this proof.
Since the distribution $Q$ is constant on conjugacy classes, it can be written uniquely as a sum of irreducible characters and we determine its coefficients $a_t$ with respect to the characters $\chi_{V_t}$ for $t=0,\dots, b$. By orthogonality of characters we have
\[a_t = \langle Q, \chi_{V_t}\rangle = \frac{1}{|S_n|}\sum_{g\in G} Q(g)\chi_{V_t}(g) = \frac{1}{|S_n|}\chi_{V_t}(\tau)\]
where $\tau$ is any $k$-cycle. By Lemma~\ref{lem: 2rows} we know that
\[a_t = \frac{1}{|S_n|}\left(\left[\binom{n-k}{a+b-t}-\binom{n-k}{a+b+1-t}\right] + \left[\binom{n-k}{t}-\binom{n-k}{t-1}\right]\right)\] 
and that $\dim(V_t)=\binom{n}{t}-\binom{n}{t-1}$. The claim now follows from Corollary~\ref{Cor: Conjugation_invariant}.
\end{proof}

\begin{figure}[h]
\centering
\includegraphics[width=5.9cm]{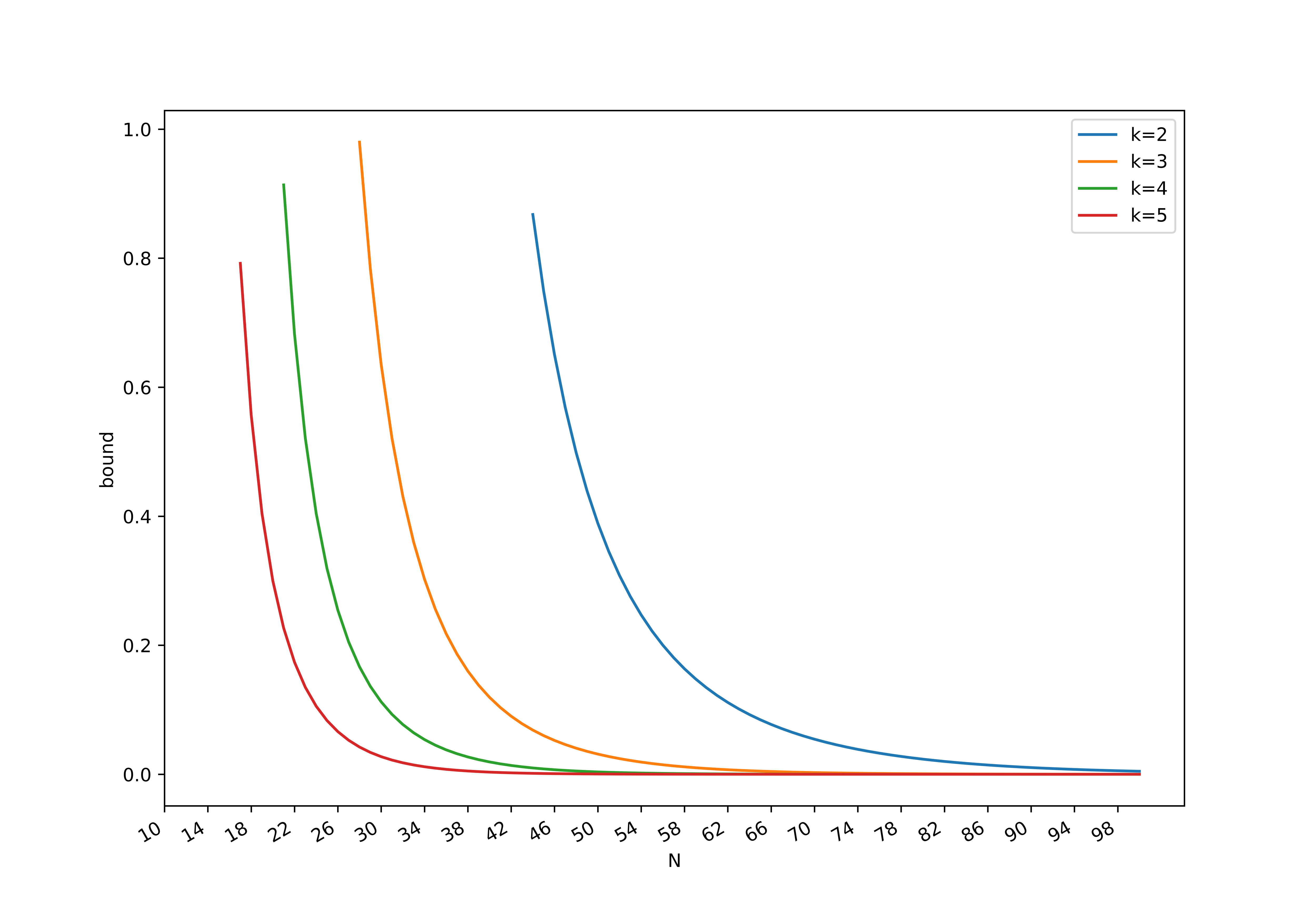} 
\includegraphics[width=5.9cm]{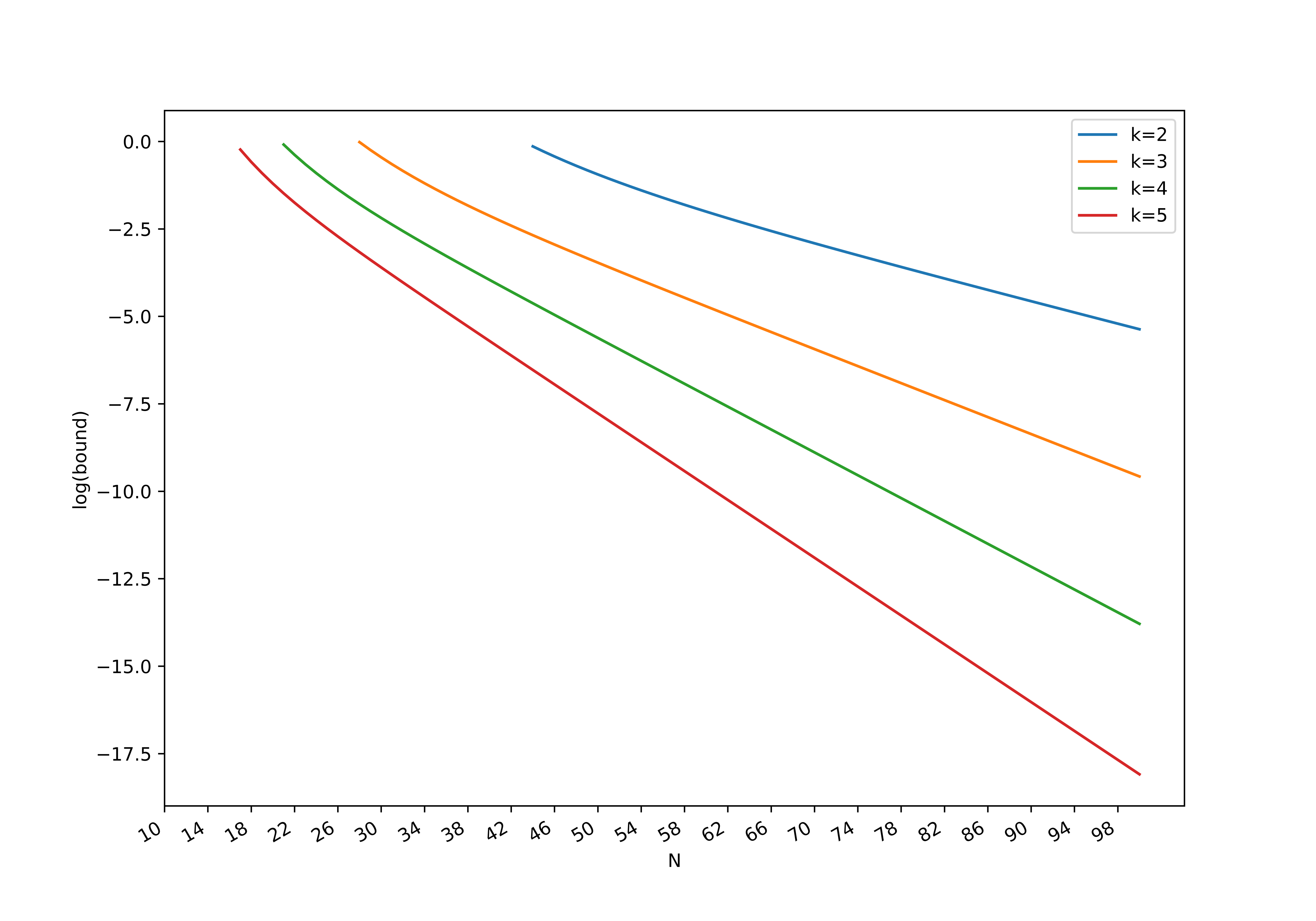}
\caption{Upper bound for total variation to uniformity for $\lambda : 26\geq 26$ with $n=52$ (the number of cards on a regular deck) and $k=2,3,4,5$ ($y$ axis in standard (left) and logarithmic (right) scales). The set $X$ has around $4.9\times 10^{14}$ elements.}
\end{figure}

\begin{remark}
The previous Corollary shows that Theorem~\ref{thm: newUB} can be applied more generally than the Upper Bound Lemma on $S_n$ since selecting a transposition uniformly at random {\it does not} mix $S_n$ (because only even transpositions can be reached in even stages), while it does converge to the uniform distribution on the homogeneous space $X$.
\end{remark}

\subsection{A concentration inequality}
\label{Section: Concentration}
In this Section we prove Theorem~\ref{Thm: Concentration} and illustrate its applicability by  analyzing random walks on traveling salesman tours.

\begin{proof}[Proof of Theorem~\ref{Thm: Concentration}] Let $\EE_{Q^N}(\bullet)$ denote the expected value with respect to the distribution of the process after $N$ stages. By the triangle inequality and the definition of total variation the following inequalities hold,
\[ \left|\EE_U(f) - \frac{1}{M}\sum_{i=1}^M f(Z_j)\right| \leq \left|\EE_U(f) -\EE_{Q^N}(f)\right|+\left|\EE_{Q^N}(f)- \frac{1}{M}\sum_{i=1}^M f(Z_j)\right|\leq \]
\[\leq\|Q^N-U\|_{\rm TV} + \left|\EE_{Q^N}(f)- \frac{1}{M}\sum_{i=1}^M f(Z_j)\right|\]
and we conclude that
\[\PP\left\{\left|\EE_U(f) - \frac{1}{M}\sum_{i=1}^M f(Z_j)\right|\geq \epsilon\right\}\leq \PP\left\{ \left|\EE_{Q^N}(f)- \frac{1}{M}\sum_{i=1}^M f(Z_j)\right|\geq \beta \right\}\]
where $\beta = \epsilon-\|Q^N-U\|_{\rm TV}$. Since $\beta>0$, the right-hand side is bounded by $2\exp\left(-M\beta^2/2\right)$ from Hoeffding's inequality for bounded random variables~\cite[Theorem 2]{Hoeffding1963} proving the claim.
\end{proof}

\begin{example} \label{example: TSP_average} Suppose $X$ is the set of tours through a fixed set of cities $1,\dots , n$ and let $\ell(x)$ be the total distance traveled in tour $x$. The set $X$ has $(n-1)!$ elements and finding a tour of least total length is the classical traveling salesman problem (TSP), of much interest in combinatorial optimization.  It is often desireable to know the {\it average cost} $\EE_{U}(f)$ of functions $f$  among all possible tours. This is especially challenging if the function depends nonlinearly on the tour. When $f$ is bounded, Theorem~\ref{Thm: Concentration} gives us a natural approach for obtaining such estimates via simulations, together with error bars on such estimates. In this example we give some nonlinear functions whose averages are of interest for the TSP and show some combinatorial techniques that can be used for obtaining the necessary total variation bounds.

Motivated by the simulated annealing approach~\cite[Section 4.4.3]{Madras} for solving the TSP, define for a fixed real number $\beta$ the probability distribution $\pi_{\beta}(x)$ on $X$, given by  the formula $\pi_{\beta}(x)=e^{-\beta \ell(x)}/C_{\beta}$, where $C_\beta(x):=\sum_{x\in X} \pi_\beta(x)$ is the associated partition function. Note that $\pi_{\beta}(x)$ assigns higher probability to to shorter tours and it is easy to see that in the limit $\beta\rightarrow\infty$ the $\pi_{\beta}(x)$-average length 
\[\overline{\ell}_\beta:=\sum_{x\in X} \ell(x)\pi_{\beta}(x)\]  converges to the length of the shortest tour. 

To estimate $\overline{\ell}_\beta$ define the functions $a_\beta(x):=\ell(x)e^{-\beta x}$, $c_\beta(x)=e^{-\beta\ell(x)}$ and the numbers $A_\beta:=\EE_{U}[a(x)]$, $C_{\beta}:=\EE_{U}[b(x)]$, and note that $\overline{\ell}_\beta = A_\beta/C_\beta$. It is natural to estimate $A_\beta$ and $C_\beta$ via their sample averages
\[
\begin{array}{ccc}
\hat{A}_\beta:=\frac{1}{M}\sum_{j=1}^M a_\beta(Z_j) & &
\hat{C}_\beta:=\frac{1}{M}\sum_{j=1}^M c_\beta(Z_j)\\
\end{array}
\]
which are easily computable with simulations. The convexity of the function $y/x$ in the positive orthant, together with Theorem~\ref{Thm: Concentration}, now imply the following Corollary which gives error bounds on these estimates. In the expressions below, $D$ is any upper bound on the length of tours (for instance the sum of the $n$ largest pairwise distances among cities). 

\begin{corollary} \label{Cor: Annealing}Let $\epsilon,\eta>0$ be real numbers. Choose $N$ large enough so that
\[\|Q^{(N)}-U\|_{\rm TV}< \frac{\epsilon e^{-2\beta D}}{D^2}\] 
and choose $M$ large enough so that 
\[2\exp\left(-\frac{M\left(\frac{\epsilon e^{-2\beta D}}{D^2}-\|Q^{(N)}-U\|_{\rm TV}\right)^2}{2}\right)<\eta.\]
If we use $M$ independent samples of the $N$-th stage of the random walk defined by the distribution $Q$ to compute the estimates $\hat{A}_\beta, \hat{C}_\beta$, then the following inequalities hold with probability at least $1-2\eta$
\[ -2\epsilon+\frac{\hat{A}_\beta}{\hat{C}_\beta}\leq \overline{\ell}_\beta=\frac{A_\beta}{C_\beta}\leq\frac{\hat{A}_\beta}{\hat{C}_\beta}+2\epsilon.\]
\end{corollary}

In order to use the previous Corollary one needs good total variation bounds. For conjugation-invariant probability distributions $Q$ such bounds follow from Corollary~\ref{Cor: Conjugation_invariant} provided we have good estimates of the involved multiplicities. 
We now illustrate such multiplicity calculations  assuming, for simplicity, that $n$ is a prime number. The stabilizer of the cycle $(1,2,\dots, n)$ is the cyclic subgroup $\ZZ/n\ZZ$ generated by the cycle $(1,2\dots, n)$ and, by the primality of $n$, this subgroup is contained in only two conjugacy classes of $S_n$, namely that of the identity and that of a single $n$-cycle $c$. By Remark~\ref{Rmk: Frobenius_Reciprocity}, we conclude that for any partition $\lambda$ of $n$ we have
\[m(S^{\lambda},\CC X) = \frac{1}{n}\left(\dim(S^{\lambda})+(n-1)\chi_{S^{\lambda}}(c)\right).\]
The value of the character $\chi_{S_{\lambda}}(c)$ when $c$ is an $n$-cycle is easily computed from the Murnaghan-Nakayama rule~\cite[Problem 4.45]{FH}, which states that
\[\chi_{S_{\lambda}}(c) = \sum_\mu (-1)^{r(\mu)}\]
where the sum is taken over the hooks $\mu$ of length $n$ contained in $\lambda$ and the integer $r(\mu)$ is defined as one less than the number of rows of $\mu$. Since $\lambda$ is a partition of $n$ it contains a hook of length $n$ if and only if $\lambda$ is itself a hook. We conclude that for the trivial and alternating representations the equalities $m(S^{1^n}, \CC X)=m(S^{(n)},\CC X)=1$ hold and that for every other hook with horizontal and vertical legs of lengths $0<a,b < n$ with $n=1+a+b$ the inequalities 
\[\frac{m(S^{\lambda}, \CC X )}{\dim(S^{\lambda})}\leq \frac{1}{n}\left(1 + \frac{n-1}{\binom{n-1}{a}}\right)\leq \frac{2}{n}\] are satisfied, 
where the middle inequality comes from the hook length formula~\cite[Formula 4.12]{FH}. Combining the previous observations we conclude that the inequality $\frac{m(S^{\lambda}, \CC X )}{\dim(S {\lambda})}\leq \frac{2}{n}$ holds for every nontrivial partition, except $1^n$, allowing us to leverage existing mixing rate estimates for conjugation-invariant random walks on $S_n$ to obtain new mixing rate estimates for $X$. Applying this to the transposition walk estimates from~\cite[Theorem 5]{D} we obtain the following Corollary which provides the total variation bounds needed for Corollary~\ref{Cor: Annealing}.

\begin{corollary} For a prime number $n$ let $X$ be the set of tours of cities $1,\dots, n$, and let $Q$ be the probability distribution on $S_n$ given by $Q(id)=1/n$ and $Q(\tau)=2/n^2$ for all transpositions $\tau$. If $q=\sum_{g\in S_n} Q(g)e_g$ then the following inequality holds for every initial state $x_0\in X$ and every stage $N=\frac{n\log(n)}{2}+cn$ for $c>0$
\[\| q^{(N)}\cdot x_0 -\overline{u}\|_{\rm TV}^2\leq \frac{2a^2e^{-4c}}{n} + \frac{\left(2/n-1\right)^{2N}}{4}\]  
where $a$ is a universal constant (i.e.,, independent of the values of $n$ and $N$).
\end{corollary}
\end{example}

\section{Switched random walks}
\label{Sec: 4}
If $Q_1,\dots, Q_m$ are probability distributions on $G$ and $X$ is a $G$-set then a choice of initial state $x_0\in X$ determines a dynamical system which we call a {\it switched random walk} on $X$. The switched random walk is a family of random variables $h^{(w)}_k$ as $w$ ranges over all words $w_1w_2\dots $ with $w_i\in [m]$ and $k\in \NN$ which describe the state of our system at time $k$ when the mixing strategies $Q_i$ have been switched according to the word $w$ (i.e.,, where strategy $Q_{w_i}$ has been used at stage $i$ for $i\leq k$). More formally, the switched random walk starting at $x_0\in X$ is constructed as follows:

\begin{enumerate}
\item Sample $[m]\times \NN$ independent elements of $G$ with $g_{i,j}$ having distribution $Q_i$. 
\item For each word $w_1w_2\dots w_N$ of length $N$ having letters $w_i\in [m]$ and $k=1,\dots , N$ let $h_k^{w}:=g_{w_N,N}\dots g_{w_2,2}g_{w_1,1}(x_0)$.

\end{enumerate}

We say that the switched random walk {\it converges to the uniform distribution} if $X$ gets evenly mixed as time passes, regardless of the initial state $x_0\in X$ and the order in which the $Q_i$ have been chosen. Quantitatively, this means that
\[\lim_{N\rightarrow \infty} \left(\max_{x_0\in X, w: |w|=N} \|q^{w}\cdot e_{x_0}-\overline{u}\|_{\rm TV} \right) = 0\] 
where $q^{w}$ is defined as the product $q^{w_1}q^{w_2}\dots q^{w_n}$ where $q^{j}:=\sum_{g\in G} Q_j(g)e_g$, so $q^{w}\cdot e_{x_0}$ encodes the distribution of $h^{w}_N$.

In this Section we address the problem of convergence of switched random walks. We define the {\it Fourier joint spectral radius} relative to $X$ of a set of distributions $Q_1,\dots, Q_m$, denoted 
by $\overline{\omega}_X(Q_1,\dots, Q_m)$ and prove Theorem~\ref{Thm: Fourier_jsrs}, which shows that this quantity captures the long-term behavior of switched random walks (see Section~\ref{Sec:Fourier-jsrs}). The effective estimation of Fourier jsrs is discussed in the final section~\ref{Sec: Estimation_Fourier_jsrs}.

\subsection{Fourier joint spectral radii.}
\label{Sec:Fourier-jsrs}

If $A_1,\dots, A_m$ are a set of $n\times n$ matrices with complex entries, then their joint spectral radius, introduced by Rota and Strang in ~\cite{RS}, is defined as
\[{\rm jsr}(A_1,\dots, A_m):=\lim_{N\rightarrow \infty} \max_{w: |w|=N}\|A_{w_1}A_{w_2}\dots A_{w_N}\|^{\frac{1}{N}}.\]
It is known that this limit always exists and that its value is independent of the chosen matrix norm. 

\begin{definition}
The {\it Fourier joint spectral radius} of the distributions $Q_1,\dots, Q_m$ on $G$ relative to $X$ is the number 
\[\overline{\omega}_X(Q_1,\dots, Q_m) := \max_{{\rm triv}\neq \rho_j\in \CC X } \left\{{\rm jsr}\left(\widehat{Q_1}(\rho_j),\dots, \widehat{Q_m}(\rho_j)\right)\right\}\] 
where the maximum is taken over the nontrivial irreducible representations $\rho_j$ of $G$ which appear in the module $\CC X$ and ${\rm jsr}$ denotes the joint spectral radius of a set of matrices. 
\end{definition}

\begin{proof}[Proof of Theorem~\ref{Thm: Fourier_jsrs}] 

Since any two norms in a finite-dimensional vector space are equivalent we know that there exist positive constants $C_1,C_2$ such that the following inequality holds for every word $w$ of length $N$ and $x_0\in X$
\[C_1 \|q^{(w)}\cdot e_{x_0} -\overline{u}\|_2\leq  \|q^{(w)}\cdot e_{x_0} -\overline{u}\|_1\leq C_2 \|q^{(w)}\cdot e_{x_0} -\overline{u}\|_2.\]
and therefore the equality in the Theorem is equivalent to 
\[\lim_{N\rightarrow \infty} \left(\max_{x, w:|w|=N}\|q^{(w)}\cdot e_{x_0} -\overline{u}\|_2\right)^{\frac{1}{N}}=\overline{\omega}_X\left(Q_1,\dots, Q_m\right).\]
For any word $w$ of length $N$ the inequality 
\[\max_{x_0\in X} \|q^{(w)}\cdot e_{x_0} -\overline{u}\|_2 \geq \sqrt{\frac{\sum_{x\in X} \|q^{(w)}\cdot e_x -\overline{u}\|_2^2}{|X|}} =\frac{\|q^{(w)}-\overline{u}_G\|_{\rm Fb}}{\sqrt{|X|}}\]
holds and by the submutiplicativity of the Frobenius norm the inequality 
\[\max_{x_0\in X} \|q^{(w)}\cdot e_{x_0} -\overline{u}\|_2\leq \|q^{(w)}-\overline{u}_G\|_{\rm Fb}\]
holds. The equality of the Theorem is therefore equivalent to
\[\lim_{N\rightarrow \infty} \left(\max_{w:|w|=N}\|q^{(w)} -\overline{u}_G\|_{\rm Fb}\right)^{\frac{1}{N}}=\overline{\omega}_X\left(Q_1,\dots, Q_m\right).\]
By Lemma~\ref{lem: orto} and the orthogonal invariance of the Frobenius norm we know that for every word $w$ of length $N$ the following equality holds

\[\|q^{(w)}-\overline{u}_G\|_{\rm Fb}^2 = \sum_{V_j\neq {\rm triv}} m(V_j,\CC X) \|\hat{Q}(\rho_j)_{w_1}\hat{Q}(\rho_j)_{w_2}\dots \hat{Q}(\rho_j)_{w_N}\|^2_{\rm Fb}\]

Taking $N$-th roots on both sides and letting $R$ denote the number of irreducible representations of $G$ appearing in $\CC X$ we obtain the inequality
\begin{tiny}
\[\|q^{(w)}-\overline{u}_G\|_2^{\frac{1}{N}}\leq  \left(R\max_{\rho_j\neq {\rm triv}} m(V_j, \CC X)\right)^{1/2N}  \left(\max_{{\rm triv}\neq \rho_j\in \CC X}\max_{w: |w|=N}\|\hat{Q}(\rho_j)_{w_1}\hat{Q}(\rho_j)_{w_2}\dots \hat{Q}(\rho_j)_{w_N}\|^2_{\rm Fb}\right)^{1/2N} \]
\end{tiny}
and therefore the inequality
\begin{tiny}
\[\max_{w: |w|=N}\|q^{(w)}-\overline{u}\|_2^{\frac{1}{N}}\leq  \left(R \max_{\rho_j\neq {\rm triv}} m(V_j, \CC X)\right)^{1/2N}  \left(\max_{{\rm triv}\neq \rho_j\in \CC X}\max_{w: |w|=N}\|\hat{Q}(\rho_j)_{w_1}\hat{Q}(\rho_j)_{w_2}\dots \hat{Q}(\rho_j)_{w_N}\|^2_{\rm Fb}\right)^{1/2N} \]
\end{tiny}
which, letting $N\rightarrow \infty$ implies that 
\[\lim_{N\rightarrow \infty} \left(\max_{w: |w|=N} \|q^{(w)}-\overline{u}\|_{2}\right)^{1/N} \leq  \overline{\omega}_X(Q_1,\dots, Q_m)\] 

For the opposite inequality note that for every irreducible representation $V_t$ appearing in $\CC X$ and every word $w$ of length $N$ we have 
\[\|\hat{Q}(\rho_t)_{w_1}\hat{Q}(\rho_t)_{w_2}\dots \hat{Q}(\rho_t)_{w_N}\|_{\rm Fb}\leq \|q^{(w)}-\overline{u}_G\|_{\rm Fb}\]
and therefore taking $N$-th roots, maximizing over $w$, and letting $N\rightarrow \infty$, we have
\[{\rm jsr}\left(\hat{Q_1}(\rho_t),\dots,\hat{Q_m}(\rho_t)\right)\leq \lim_{N\rightarrow \infty}\max_{w: |w|=N}\|q^{(w)}-\overline{u}_G\|_{2}^{1/N}.\]
We conclude that the right-hand side is bounded below by $\overline{\omega}_X(Q_1,\dots, Q_m)$ proving the equality in the Theorem.
Since the total variation distance between two probability distributions is bounded by one, the equality 
\[\lim_{N\rightarrow \infty} \left(\max_{x, w:|w|=N}\|q^{(w)}\cdot e_{x_0} -\overline{u}\|_{\rm TV}\right)^{\frac{1}{N}}=\overline{\omega}_X\left(Q_1,\dots, Q_m\right)\]
implies that $\overline{\omega}_X\left(Q_1,\dots, Q_m\right)\leq 1$ and we will show that $Q_1,\dots, Q_m$ has the adversarial mixing property if and only if the strict inequality holds.
If $\overline{\omega}_X(Q_1,\dots, Q_m)<1$ and $\alpha$ is any real number with $\overline{\omega}_X(Q_1,\dots, Q_m)<\alpha<1$ then there exists an integer $N_0$ such that for every initial $x_0\in X$ and word $w$ of length $N\geq N_0$ we have
\[\|q^{(w)}\cdot e_{x_0}-\overline{u}\|_{\rm TV}\leq \alpha^N,\]
proving the convergence to the uniform distribution, since $\alpha^N$ converges exponentially to zero. Conversely, if $\overline{\omega}_X(Q_1,\dots, Q_m)=1$, then there exists a representation $\rho_t$ appearing in $\CC X$ such that ${\rm jsr}\left(\hat{Q_1}(\rho_t),\dots, \hat{Q_m}(\rho_t)\right)=1$. By~\cite[Theorem 2]{Breuillard}, $\limsup\max_{w}\Lambda(\prod_{w_i}\hat{Q_{w_i}}(\rho_t))=1$, where $\Lambda(A)$ denotes the magnitude of the largest eigenvalue of the matrix $A$. As a result, given $\epsilon>0$ there exists a sequence of integers $n_j$ and words $w^{j}$ of length $n_j$ such that $\hat{Q_{w^j_1}}(\rho_t)\dots \hat{Q_{w^j_{n_j}}}(\rho_t)$ has an eigenvalue of size at least $(1-\epsilon)$ and therefore its Frobenius norm is at least $1-\epsilon$. By Corollary~\ref{Cor: good_and_bad} for every such word there exists an initial state $x_j\in X$ such that
\[\|q^{(w^{j})}\cdot e_{x_j}-\overline{u}\|_{\rm TV}\geq \sqrt{\frac{(1-\epsilon)^2}{4|X|}}\]
Since $X$ is finite, there is an initial state $x^*$ which appears infinitely many times among the $x_j's$, and we conclude that the switched random walk determined by $Q_1,\dots, Q_m$ starting from $x^*$ does not converge.
\end{proof}

At this point the skeptical reader may wonder whether the theory is trivial in the sense that the equality $\overline{\omega}_X(Q_1,\dots,Q_m)=\max_j\left( \overline{\omega}_X(Q_j)\right)$ holds in general (case in which switching the random walk can never be worst than permanently using some of its defining distributions). The following example shows that this is not the case even for $3\times 3$ doubly stochastic matrices. 

\begin{example} [Non-triviality] \label{example: non-triviality} Consider the following two probability distributions in $S_3$
\[
\begin{tabular}{c|cccccc}
	& e & (23) & (12) & (123) & (132) & (13)\\
\hline
$Q_1$ & 2/8 & 1/8 &  1/8   & 2/8 & 1/8 & 1/8 \\ 
$Q_2$ & 1/8 & 1/8 &  1/8   & 2/8 & 1/8 & 2/8 \\ 
\end{tabular}
\]
Their action in the permutation module $M^{(2,1)}$ is given by the matrices
\begin{tiny}
\[
M_1 =  \left(\begin{array}{ccc}
3/8 & 1/4 & 3/8\\
3/8 & 3/8 & 1/4\\
1/4 & 3/8 & 3/8\\
\end{array}\right)\text{ , }
M_2 = \left(\begin{array}{ccc}
1/4 & 1/4 & 1/2 \\
3/8 & 3/8 & 1/4 \\
3/8 & 3/8 & 1/4 \\
\end{array}\right)
.\]
\end{tiny}
Their Fourier jsr relative to $M^{(2,1)}$ is equal to the jsr of their Fourier transforms in the representation $S^{(2,1)}$, namely the matrices:
\[
N_1 =  \left(\begin{array}{cc}
0.0625  &  0.108253\\
-0.108253 & 0.0625\\
\end{array}\right)\text{ , }
N_2 = \left(\begin{array}{cc}
-0.125    &  0\\
-0.216506 & 0\\
\end{array}\right)
.\]
Their spectral radii $\Lambda$ satisfy $\Lambda(N_1)=0.125$, $\Lambda(N_2)= 0.125$ and $\Lambda(N_1N_2)=0.03125$.
As a result
$\Lambda(N_1N_2)>\max\left(\Lambda(N_1),\Lambda(N_2)\right)^2$ and therefore
\[\overline{\omega}_{M^{(2,1)}}(M_1,M_2)\geq \Lambda(N_1N_2)^{1/2} > \max_i\left(\Lambda(N_i)\right) = \max\left(\overline{\omega}_{M^{(2,1)}}(M_1),\overline{\omega}_{M^{(2,1)}}(M_2)\right).\]
\end{example}
\begin{remark} The example above notwithstanding, there are cases beyond the trivial situation of commuting matrices where switching {\it does not} make mixing more elusive. For instance, if the $Q_i$ are {\it symmetric distributions}, in the sense that $Q_i(g)=Q_i(g^{-1})$ for all $g\in G$, then the matrices $\hat{Q}_i(\rho_j)$ are hermitian and in particular, their spectral radius coincides with their operator norm $\|\hat{Q}_i(\rho_j)\|$ (matrices with this property are called radial and have been classified~\cite{radial}). Since the operator norm is submultiplicative, for any word $w$ of length $N$ we have the inequality
\[\|\hat{Q}_{w_i}(\rho_j)\cdots \hat{Q}_{w_N}(\rho_j)\|^{1/N}\leq \max_t \|\hat{Q}_{t}(\rho_j)\|\] which implies that the equality 
\[\overline{\omega}_X(Q_1,\dots, Q_m)=\max_j \left(\overline{\omega}_X(Q_j)\right)\]
holds for every $G$-set $X$ and {\it for symmetric distributions} $Q_1,\dots, Q_m$ or, more generally, for distributions whose Fourier transforms are radial.
\end{remark}

\begin{remark} By Birkhoff's Theorem, the convex hull of permutation matrices coincides with the set of doubly stochastic matrices. It follows that doubly stochastic matrices are precisely the possible random walks on $M^{(n-1,1)}$ induced by a distribution $Q$ on $S_n$. Since $M^{(n-1,1)}={\rm triv} \oplus S^{(n-1,1)}$ it follows that for a single distribution the number $\overline{\omega}_{M^{(n-1,1)}}(Q)$ coincides with the SLEM (Second largest eigenvalue in magnitude) of the chain studied in~\cite{DP, DP2} for the design of fast-mixing chains. We can think of Fourier spectral radii as a generalization of this quantity for several distributions and arbitrary symmetries.
\end{remark}

\subsection{Estimation of Fourier jsrs}
\label{Sec: Estimation_Fourier_jsrs}
The computation of the jsr of a given set of matrices is a surprisingly difficult problem. As mentioned in the introduction, it is known to be {\it undecideable} whether the jsr of a pair of matrices is bounded by one and it is unknown whether checking if it is strictly bounded by one is decideable. Nevertheless the seminal work of Parrilo and Jabjabadie~\cite{PJ}, Ahmadi and Jungers~\cite{AJ1} and Ahmadi, de Klerk and Hall~\cite{PN}, among others, has provided us with sum-of-squares algorithms which are capable of approximating jsrs to arbitrary accuracy (albeit at an often significant computational effort which cannot be predicted in advance as the undecideability results show). In this section we extend the results of~\cite{PN} to polynomial norms expressible via hermitian sums of squares, allowing us to estimate join spectral radii for matrices with complex entries, the case of interest for the computation of Fourier jsrs. Note that the extension is not completely trivial, since a norm on the underlying real vector space of $\CC^n$ is not generally a complex norm because the equality $\|\lambda x\|=|\lambda|\|x\|$ needs to hold for arbitrary {\it complex} numbers $\lambda$.

We begin by explaining the general approach for the estimation of jsrs via sums of squares. Recall that the jsr of a set $A_1,\dots, A_m$ of $n\times n$ matrices with complex entries is a limit which can be computed using any matrix norm. If we use a matrix norm $\|\bullet\|_{\rm op}$ which is induced by a norm $\|\bullet\|$ on vectors in $\CC^n$, then the submultiplicativity of induced norms implies that the inequality
\[{\rm jsr}(A_1,\dots, A_m)\leq \max_j \|A_j\|_{\rm op}\]
holds. A basic result of Rota and Strang~\cite{RS} states that such inequalities give arbitrarily good estimates for jsr's in the sense that
\[{\rm jsr}(A_1,\dots, A_m) = \inf_{\|\bullet\|}\left(\max_j \|A_j\|_{\rm op}\right)\]
as the infimum runs over all norms in $\CC^n$.

Over the real numbers we know from results of~\cite{PN} that arbitrary norms can be uniformly approximated by polynomial norms (i.e.,, by norms of the form $V(x)=f^{1/{2d}}(x)$ where $f$ is some sum-of-squares form of degree $2d$), proving that the optima over polynomial norms $V$ of increasing degrees satisfying $V(A_ix)\leq \gamma V(x)$ would eventually prove that a real number $\gamma$ is an upper bound for the jsr of $A_1,\dots,A_m$ if this is indeed the case. The following Lemma extends this over the complex numbers, proving that all norms on a complex vector space can be approximated via norms defined by {\it hermitian} sums of squares.

\begin{theorem}[Complex polynomial norms] \label{thm: CPN} Let $\|\bullet\|$ be a norm in $\CC^n$. There exists a sequence of hermitian sums-of-squares
\[F_{2d}(z):=\sum_{i=1}^{N(2d)} w_i |\langle z,y_i\rangle|^{2d}\]
for some $N(2d)\in \NN$, points $y_i\in \CC^n$ and real positive weights $w_i$ for $i=1,\dots, N(2d)$ which satisfy the following properties:
\begin{enumerate}
\item The function $n_{2d}(z):= F_{2d}(z)^{\frac{1}{2d}}$ is a norm in $\CC^d$.
\item The inequality $n_{2d}(z)\leq \|z\|$ holds for $z\in \CC^d$.
\item The sequence $n_{2d}(z)$ converges to $\|z\|$ uniformly on compact subsets of $\CC^d$.
\end{enumerate}
\end{theorem}
\begin{proof} Let $B\subseteq \CC^n$ be the unit ball for $\|\bullet\|$ and let $B^{\circ}$ be its polar set
\[B^{\circ}:=\{y\in \CC^n: \|y\|_*\leq 1\}\]
where $\|y\|_*:=\sup_{x\in B} |\langle x,y\rangle|$ is the norm dual to $\|\bullet\|$. Let $dy$ denote the Lebesgue measure in $\CC^n$ and define the measure  $\mu(A):=\frac{1}{{\rm Vol}(B^{\circ})}\int_{B^{\circ}\cap A} dy$. Since $\mu$ has compact support, the generalized Tchakaloff Theorem--complex case of Curto and Fialkow~\cite[Theorem 3.1]{CF}, implies that for every degree $2d$ there exist $N_i(2d)\in \NN$, points $y_i\in B^{\circ}$ and positive weights $w_i$ for $i=1,\dots, N_i(2d)$ such that for every polynomial $f(z,\overline{z})$ of degree $2d$ the equality
\[\int_{\CC^n} f(z,\overline{z})d\mu(z) = \sum_{i=1}^{N(2d)} w_i f(y_i,\overline{y_i})\]
holds. In particular, for every $x\in \CC^n$ we have
\[F_{2d}(x):=\int_{\CC^n} |\langle x,y\rangle|^{2d}d\mu(z)=\frac{1}{{\rm Vol}(B^{\circ})}\int_{B^{\circ}} |\langle x,y\rangle|^{2d} dy = \sum_{i=1}^{N(2d)} w_i |\langle x,y_i\rangle|^{2d}\]
proving that $F_{2d}(x)$ is a sum of hermitian squares. Furthermore if $0=F_{2d}(\alpha)$ the integral expression implies that $\langle \alpha, y\rangle=0$ for $y\in B^{\circ}$ so $\|\alpha\|=0$. As a result, the linear map $\phi: \CC^n\rightarrow \CC^{N(2d)}$ sending $x$ to $(\langle x,y_i\rangle)_i$ is injective. It follows that the function $n_{2d}(x):=F_{2d}(x)^{\frac{1}{2d}}$, which is obtained from the $\ell_{2d}$-norm $\left(\sum |a_i|^{2d}\right)^{2d}$ in $\CC^{N(2d)}$ by composition with the injective linear map $\phi$, is automatically a norm on $\CC^d$, proving $(1)$.
For $(2)$ note that the inequality $\langle x,y\rangle\leq \|x\|\|y\|_*$ bounds the integral form of $F_{2d}(x)$ from above by $\|x\|^{2d}$, yielding the inequality $n_{2d}(x)\leq \|x\|$ for all $x\in \CC^n$. $(3)$ Let $S$ (resp $S^*$) be the unit sphere for the norm $\|\bullet\|$ (resp. for the norm $\|\bullet\|_*$). Given $\epsilon>0$, the compactness of $S$ implies that there are finitely many centers $a_1,\dots a_M$ in $S$ such that balls of norm $\|\bullet\|$ with radius $\epsilon$ centered at the $a_i$ cover $S$. For each $i$ let $b_i\in S^*$ be such that $\langle a_i,b_i\rangle = 1$ and define $B_i^{\circ}:=\{y\in B^{\circ}: \|y-b_i\|_*\leq \epsilon\}$. We claim that for every $\epsilon>0$ there exists $d$ such that $n_{2d}(x)\geq 1-3\epsilon$ for every $x'\in S$ simultaneously, proving $(3)$. To verify this claim take $x'\in S$ and assume $i$ is such that $\|x'-a_i\|<\epsilon$. For every $y'\in B_i^{\circ}$ we have 
\[ |\langle x',y'\rangle|=|\langle a_i,b_i\rangle+ \langle x',y'\rangle-\langle a_i,b_i\rangle|\geq 1-|\langle x',y'\rangle-\langle a_i,b_i\rangle|\geq 1-2\epsilon\] 
where the last inequality holds because
\[|\langle x',y'\rangle-\langle a_i,b_i\rangle|= |\langle x',y'-b_i\rangle+\langle x'-a_i,b_i\rangle|\leq \|y'-b_i\|_*+\|x'-a_i\|\leq 2\epsilon.\]
We conclude that for every $x'$ with $\|x'\|=1$ we have
\[n_{2d}(x')\geq (1-2\epsilon)\left(\min_{i=1,\dots, m}\frac{{\rm Vol}(B_i^{\circ})}{{\rm Vol}(B^{\circ})}\right)^{\frac{1}{2d}}\]
and the right hand side can be made larger than $1-3\epsilon$ by choosing a sufficiently large $d$.
\end{proof}

The previous Theorem shows that norms defined by sums of powers of hermitian squares of linear forms are sufficiently general so as to approximate all norms. Now we will consider a relaxation which has the advantage of being expressible via Hermitian semidefinite programming. To this end, let $\gamma\geq 0$ be any real number and let $L(z)$ be a hermitian polynomial in the variables $z_1,z_2,\dots, z_n$, and assume that they satisfy:

\begin{enumerate}
\item $L(z)$ is a hermitian sum-of-squares of forms of degree $d$ in $z_1,\dots, z_n$. In particular $L$ is real-valued in $\CC^n$.
\item There exists $\epsilon>0$ such that $L(z)\geq \epsilon \|x\|^{2d}$.
\item For every $(z,w)\in \CC^{n}\times \CC^n$ the inequality $w^*H_L(z)w\geq 0$ holds, where $H_L(z)$ denotes the (hermitian) Hessian matrix of $L(z)$. 
\item The inequalities $L(A_jz)\leq \gamma^{2d}L(z)$ hold.
\end{enumerate}

Theorem~\cite[Theorem 2.1]{PN} and condition $(1)$, which guarantees the correct behavior for scalar multiplication, imply that $V(z):=L(z)^\frac{1}{2d}$ is a complex norm. By $(4)$, its induced operator norm proves that ${\rm jsr}(A_1,\dots, A_m)\leq \gamma$. 
Conversely, if ${\rm jsr}(A_1,\dots, A_m)< \gamma$, then there exists an integer $d$ such that $F_{2d}(z)$ from Theorem~\ref{thm: CPN} satisfies items $(1)$,$(3)$ and $(4)$ with strict inequalities, so the set of $L$'s satisfying the above inequalities strictly are able to guarantee 
that ${\rm jsr}(A_1,\dots, A_m)< \gamma$ when this is the case. Quillen's positivity Theorem~\cite[Theorem 9.50]{BPT}, which says that every strictly positive byhomogeneous form $f(z,\overline{z})$ becomes a sum of hermitian squares (HSOS) when multiplied by a $\|z\|^{2r}$ for some sufficiently large integer $r$, can now be used to construct the desired hierarchy of hermitian semidefinite programs. More precisely, we have proven

\begin{corollary} If ${\rm jsr}(A_1,\dots,A_m)<\gamma$, then there exist integers $r,d>0$ and a real number $\epsilon>0$ such that the following hermitian semidefinite program is feasible:
\begin{enumerate}
\item $L(z)$ is a hermitian sum-of-squares of forms of degree $2d$ in $z_1,\dots,z_n$.
\item There exists $\epsilon>0$ such that $\|z\|^{2r}\left(L(z)-\epsilon \|x\|^{2d}\right)$ is HSOS.
\item The function $\|(z,w)\|^{2r}w^*H_L(z)w$ is HSOS, where $H_L(z)$ denotes the (hermitian) Hessian matrix of $L(z)$. 
\item The functions $\|z\|^{2r}\left(\gamma^{2d}L(z)-L(A_jz)\right)$ are HSOS.
\end{enumerate}
Conversely, any $L(z)$ satisfying the conditions above defines a norm $V(z):=L(z)^{\frac{1}{2d}}$ which provides a proof that ${\rm jsr}(A_1,\dots,A_m)\leq \gamma.$
\end{corollary}

Our final example illustrates the sum-of-squares approach for estimating Fourier jsrs. The values of the characters of the symmetric group are rational numbers and therefore we can limit ourselves to real sums-of-squares when computing Fourier jsrs for distributions on $S_n$.

\begin{example} 

\begin{figure}[h]
\centering
\includegraphics[width=8.9cm]{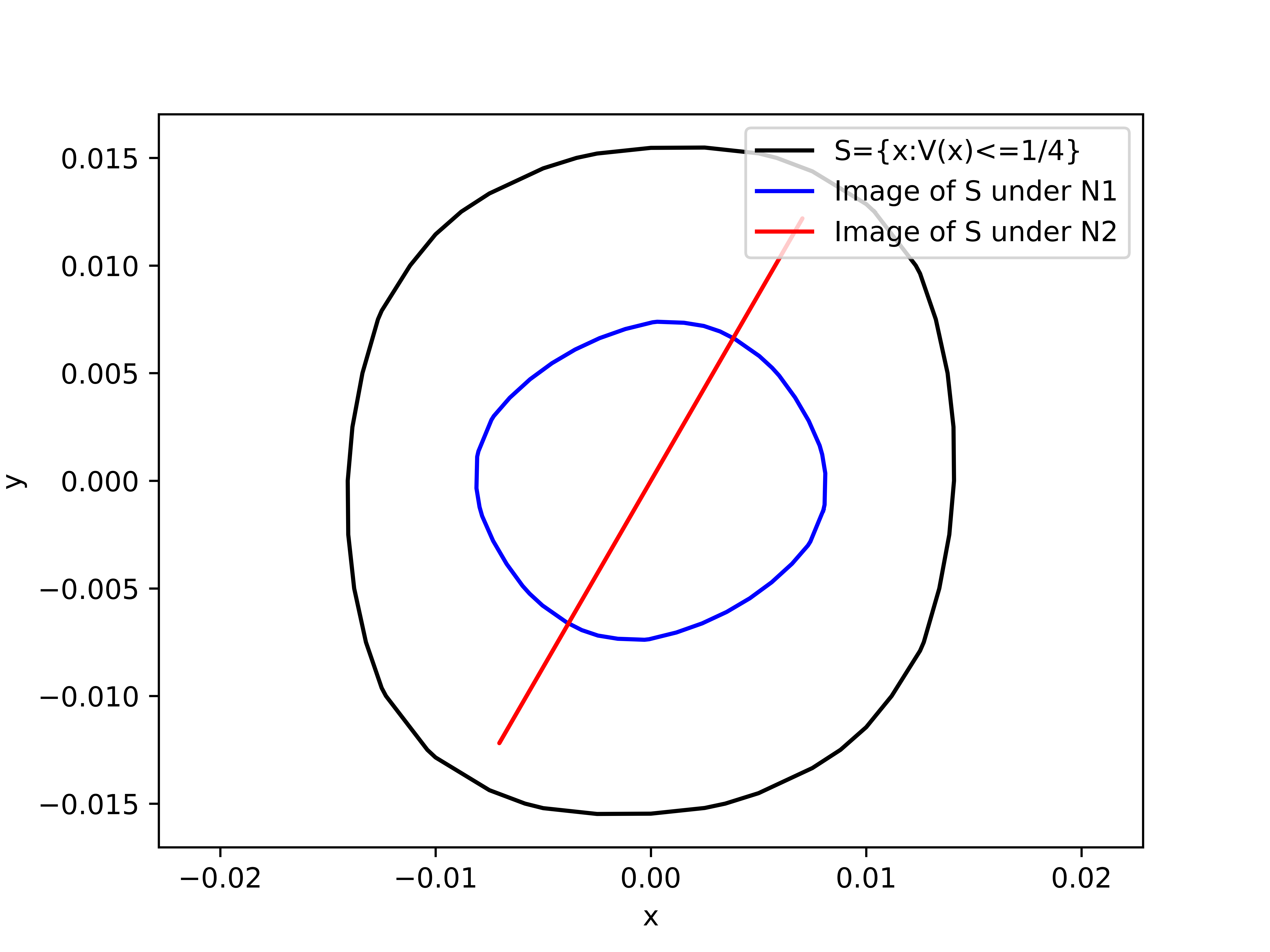} 
\caption{The ball of $V$ with radius $1/4$ contains the images under $N_1$ and $N_2$ of the unit ball of $V$.}
\end{figure}

We wish to estimate the Fourier  spectral radius of the distributions $Q_1,Q_2$ defined in Example~\ref{example: non-triviality} relative to $\CC S_3$. In the sign representation $\rho_{sgn}$ the fourier transforms satisfy $\hat{Q}_1(\rho_{\rm sgn})=1/4$ and $\hat{Q}_2(\rho_{\rm sgn})=0$, so we know that
\[\overline{\omega}(Q_1,Q_2)=\max\left(\frac{1}{4}, {\rm jsr}(N_1,N_2)\right)\]
where $N_1$ and $N_2$ are the fourier transforms on the irreducible representation $S^{(2,1)}$, computed explicitly in Example~\ref{example: non-triviality}.
To estimate the jsr of $N_1$ and $N_2$ we solve the optimization problem above. This problem constructs a polynomial $F(x)$ of degree $2d=4$ such that $V(x):=F(x)^{\frac{1}{2d}}$ is a norm which certifies that the inequality ${\rm jsr}(N_1,N_2)\leq 0.25$ holds, proving that $\overline{\omega}(Q_1,Q_2)=1/4$.
\end{example}

\end{document}